\documentclass[11pt]{article}
\usepackage{mathtools}

\usepackage{amsmath}
\usepackage{makeidx}
\makeindex
\usepackage{amssymb}
\usepackage{amsthm}

\newtheorem{thm}{Theorem}[section]
\newtheorem{lema}{Lemma}[section]
\newtheorem{cor}{Corollary}[section]					
\newtheorem{prop}{Proposition}[section]

\newtheorem{obs}{Remark}[section]

\newtheorem{ass}{Assumption}[section]

\def\p{^{\prime}}

\def\cN{\mathcal{N}}
\def\U{\mathrm{U}}
\def\R{\mathbb{R}}
\def\T{T}

\def\bU{\bar{\U}}
\def\E{\mathbb{E}}

\def\P{\mathbb{P}}

\def\one{{\bf 1}\hskip-.5mm}

\let\epsilon=\varepsilon

\title{\bf{A model for neural activity in the absence of external stimuli}}

\author{Aline Duarte  and Guilherme Ost \thanks{e-mail addresses: \tt{aline.duart@gmail} and \tt{guilhermeost@gmail.com}}
\vspace{.5cm} \\
{\it {Universidade de S\~ao Paulo} and GSSI - L'Aquila \thanks{The authors are PhD students from the Universidade de S\~ao Paulo and this work was developed during their partial completion doctoral program at GSSI both being fully supported by CNPq.}  }
}

\date{\today}

\begin{document}

\maketitle

\begin{abstract}
We study a stochastic process describing the continuous time evolution of the membrane potentials of finite system of neurons  in the absence of external stimuli. The values of the membrane potentials 
evolve under the effect of {\it chemical synapses}, {\it electrical synapses} and \textit{leak currents}. The evolution of the process can be informally described as follows. Each neuron spikes randomly following a point process with rate depending on its membrane potential.  When a neuron spikes, its membrane potential is immediately reset to a resting value. Simultaneously, the membrane potential of the neurons which are influenced by the spiking neuron receive an additional positive value. Furthermore, between consecutive spikes,  the system follows a deterministic motion due both to electrical synapses and leak currents. Electrical synapses push the system towards its average potential, while leak currents attract the membrane potential of each neuron to the resting value. 

We show that in the absence of leakage the process converges to an unique invariant measure, whenever the initial configuration is non null. More interesting, when leakage is present,
we prove the system stops spiking after a finite amount of time almost surely. This implies that the unique invariant measure is supported only by the null configuration.
\end{abstract}

{\it Key words} : piecewise deterministic Markov process, limiting distribution, neuronal systems, chemical synapses, electrical synapses, leak current\\

{\it AMS Classification}: 60K35, 60F99, 60J25

\section{Introduction}

We study the behavior of a finite number of interacting neurons  in the absence of external stimuli our  goal  being to determine the long-run behavior of the process. Our system is composed of  $N$ neurons whose state at time $t\geq 0$
is specified  by $\U(t)=\left(\U_1(t),\ldots \U_N(t)\right)$, with $\U(t)\in \R_{+}^N.$ For each neuron $1\leq i\leq N$ and each time $t\geq 0,$ $\U_i(t)$ represents the membrane potential of neuron $i$ at time $t$.
We consider two kinds of interactions among neurons and also a constant interplay between neurons and the environment.

More precisely the neurons interact via {\it   electrical} and {\it chemical synapses}. Electrical synapses are due to so-called \textit{gap-junction channels} between neurons which induce a constant sharing of potential, pushing the system towards its average value. By contrast, chemical synapses are point events which can be described as follows. Each neuron spikes randomly
at rate $\varphi(\U)\ge 0$ which depends on its membrane potential $\U$. The {\it spiking rate} $u\mapsto\varphi(u)$ is a non decreasing continuous function, positive at $u>0$ and both differentiable and vanishing at $0$ (in agreement with the assumption of non external stimuli).
When neuron $i$ spikes, its membrane potential is immediately reset to a resting potential $0$. Simultaneously, the neurons which are influenced by neuron $i$ receive an additional positive value to their membrane potential. This value may vary for each pair of neurons.
Moreover, in the whole time, neurons loose potential to the environment, due to \textit{leakage channels} which  pushes down the membrane potential of each neuron toward zero. This outgoing constant flow of potential is defined as {\it leak currents.} For technical details we refer the reader to 
Gersnter and Kistler \cite{Gerstner:2002:SNM:583784}.

 Our system is inspired by the one introduced by Galves and L\"ocherbach 
\cite{GalEva:13}. This model is an example of piecewise-deterministic Markov processes (PDPs) introduced by Davis \cite{Davis:84}. Such processes combine random jump events, the chemical synapses, with deterministic continuous evolutions, in our case due both to electrical synapses and leak currents.  The PDPs have been used also to model neural systems by other authors, see for instance the papers by
 Riedler, Thieullen and Wainrib \cite{Thieullen:12}, 
De Masi et al. \cite{Errico:14}, Fournier and L\"ocherbach \cite{Evafou:14} and  
Robert and Touboul
 \cite{Robert:14}.

Chemical synapses and leakage  make the system non-conservative. Moreover, there is an  evident competition between the incoming energy induced by the spikes and the outgoing energy induced by leak currents. Therefore it is natural to ask about the limiting behavior  of the system as time $t \rightarrow \infty$. The main results of the paper, presented in Theorem \ref{UstopSpiking} and Theorem 
\ref{med_inv_semleaking}, provide a complete description of the asymptotic distribution of the process.
Theorem~\ref{UstopSpiking} states that  under the presence of leakage almost surely there are only a finite number of spikes
and the system converges to an ``inactive global state''  interpreted as  ``brain sleep''.  
On the other hand, when leakage is absent, we prove that the process is {\it Harris-ergodic}, whenever the initial configuration is non null. This is the content of the Theorem \ref{med_inv_semleaking}.

Our paper is organized as follows. In Section \ref{ModelDef}, we introduce the process, prove its existence and we state the
main results, Theorem \ref{UstopSpiking} and Theorem \ref{med_inv_semleaking}. In Section 3, we prove 
Theorem \ref{UstopSpiking}, while in Section 4, we prove Theorem~\ref{med_inv_semleaking}. In Section 5, 
we briefly compare similar results recently obtained by Robert and Touboul 
 \cite{Robert:14}.

\section{Model definition and main results}
\label{ModelDef}

Let $\cN=\{1, \cdots, N\}$ be a finite set of neurons, for some fixed integer $N\geq 1$ and consider the family of synaptic weights ${\bf W}=\left(W_{i\rightarrow j}\right)_{i,j\in \cN}\in \R_+^{\binom{N}{2}}$ such that $W_{i\rightarrow i}=0$ for all $i\in\cN$. The value $W_{i\rightarrow j}$ corresponds to the value added to the membrane potential of neuron $j$ when the neuron $i$ spikes.

We consider a continuous time Markov process
$$
\U(t)=(\U_1(t),\ldots, \U_N(t)), \ t\geq 0,
$$
taking values in $\R^{N}_+$, whose infinitesimal generator is given for any smooth test function  $f:\R^{N}_+\rightarrow \R$, by
\begin{equation}
\label{geradorU}
\mathcal{L}f(u)=\! \sum_{i \in \cN} \varphi(u_i)[f(\Delta_i(u))-f(u)]
- \lambda\sum_{i\in\cN}\Bigl(\frac{\partial f}{\partial u_i}(u)[u_i-\bar{u}] 
\Bigr)
-\alpha\sum_{i\in\cN}\Bigl(\frac{\partial f}{\partial u_i}(u)[u_i] \Bigr),
\end{equation}
where, for all $i\in\cN$, $\Delta_i:\R^{N}_+\rightarrow \R_+^N$ is defined by
\begin{equation*}
(\Delta_i(u))_j=\left\lbrace
			\begin{array}{ll}
				u_j+W_{i\rightarrow j}, & \mbox{if} \ j\neq i, \\
				0, & \mbox{if} \ j=i,
			\end{array}
	\right.
\end{equation*}
$\lambda,\alpha\geq 0$ are positive parameters modelling, respectively, the strength of electrical synapses and the leakage effect, $\bar{u}=(1/N)\sum_{i=1}^N u_i$ and

\begin{ass} 
\label{a1}
$\varphi:\R_+ \rightarrow \R_+$ is a non-decreasing continuous function such that $\varphi(0)=0$, $\varphi(u)>0$ for $u>0$. 
\end{ass}

Assumption \ref{a1} implies that external stimuli are not considered. This is a consequence of the condition $\varphi(0)=0$. Moreover, from the neurobiological point of view, it is
reasonable to assume that $\varphi$ is a non-decreasing function since an addition in the membrane potential increases the probability of a spiking to occur. However, we could remove the continuity assumption on $\varphi$, imposing only a weaker condition (see final discussion).

The first term in (\ref{geradorU}) depicts how the chemical synapses are incorporated in our  model. Neurons whose potential is $u$ spike at rate $\varphi(u)$. Intuitively this means 
that any initial configuration $u\in \R_+^N$ of the membrane potentials
$$\P(\U(t)=\Delta_i(u)|\U(0)=u)=\varphi(u_i)t+o(t), \quad \mbox{as} \ t \rightarrow 0.$$
Thus, the function $\varphi$ is called firing or spiking rate of the
system. 

The second and third terms in (\ref{geradorU}) represent the electrical synapses and leak currents respectively. They describe the deterministic time evolution of the system between two consecutive spikes. More specifically, in an interval of time $[a,b]$, without occurrence of spikes in the whole system, the membrane potential of neuron $i\in\cN$ obeys the following ordinary differential equation 
\begin{equation}
\label{determotion}
{\frac{d}{dt}}\U_i(t)=-\alpha \U_i(t)-\lambda(\U_i(t)-\bU(t)).
\end{equation}
Notice that the first term on the right-hand side of (\ref{determotion}) pushes simultaneously all neurons to the resting state, while the second term tends to attract the neurons to the average potential. 

Our first theorem proves the existence of the process. Its proof is a simple adaptation of Theorem 1 of \cite{Errico:14} to our case. In what follows, for any vector $u\in \R_+^N$, we write $\|u\|=\max_{i\in \cN} u_i.$

\begin{thm}
\label{existenceofU}
Let $\varphi:\R_+\rightarrow \R_+$ be any function satisfying Assumption \ref{a1}. For any $N \geq 1 $ and any $u \in \R_+^N $ there exists a unique strong Markov process $\U(t) $ taking values in $\R_+^N $ starting from $u$  whose generator  is given by \eqref{geradorU}.
\end{thm}

\begin{proof}
Let $N_i(t), \ t \ge  0$, be the simple point process on $\R_+$ which counts the
jump events of neuron $i\in \cN$ up to time $t$. Define $m_i=\sum_{j\ne i}W_{i\rightarrow j}$ and $m=\max_{i\in\cN}m_i$ and, following \cite{Errico:14}, consider the following random variable, for all $t>0$,
$$
k(t)=\sum_{i\in\cN} \int_0^t \one\{\U_i(s^-)\le 2m\}dN_i(s).
$$
The random variable $k(t)$ counts the number of spikes of neurons whose the potential is at most $2m.$
Suppose $\U_i$ fires at time $t$, in this case
\begin{equation*}
\bU(t)=\frac{1}{N}\sum_{j \neq i} \big( \U_j(t^-) + W_{i \rightarrow j} \big)
= \bU(t^-) + \frac{1}{N} \big( m_i - \U_i(t^-)\big).
\end{equation*}
Now, using the expression of $\bU(t)$ above and  adapting the proof of Theorem 1 of 
\cite{Errico:14}, we have the following inequalities for all $t>0$,
\begin{align*}
&\bU(t)\le \bU(0) + \frac{m}{N}k(t),  \ \
mN(t)\le N\bU(0)+ 2m k(t), \\ 
\intertext{where $N(t)=\sum_{i\in \cN}N_i(t)$ and} \
&\|\U(t)\|\le (N+1)\|\U(0)\|+2mk(t).
\end{align*}
Since we can bound $mk(t)$ by a Poisson process of intensity $N\varphi(2m)$, the second inequality above shows that number of jumps of the process is finite almost surely on any finite time
interval. To conclude the proof just note that the construction of the process can be achieved by gluing together trajectories given by the deterministic flow between successive jump times. This procedure is feasible since the number of jumps of the process is finite on any finite interval.
\end{proof}

Now, we shall present an elementary argument which shows that for all leakage rate $\alpha$ large enough and if the firing rate $\varphi$ is globally Lipschitz with $\varphi(0)=0$, the system goes extinct. This result was the starting point of this paper. The idea of this proof was taken from discussions with Galves and L\"ocherbach. The result is the following. In the sequel, for any $u\in \R_+^N$, we write $\P_u$ to denote the probability measure under which $\U(0)=u$ almost surely. In this way, we denote by $\E_u$ the expectation taking with respect to the probability measure $\P_u.$

\begin{thm}
\label{EvaStopSpik}
For any $N\geq 1$, $\alpha\geq 0$, $\lambda\geq 0$ and $c$-Lipschitz function $\varphi:\R_+\rightarrow\R_+$ such that $c>0$ and $\varphi(0)=0$, the following inequality holds, for all $t\geq 0$ and $u\in\R^{N}_+$,
$$\E_u[\bar{\U}(t)]\leq \bar{u}e^{t(c\alpha^*-\alpha)},$$
where $\alpha^*=\max\limits_{k\in\cN}\sum_{j\in\cN}W_{j\rightarrow k}$.
In particular, if $\alpha>c\alpha^*$, then the process goes extinct.
\end{thm}
\begin{proof}
For each $i\in\cN$, plugging $f=\pi_i$ in (\ref{geradorU}), where $\pi_i$ is the projection onto the $i$-th coordinate,  we have
\begin{align*}
{\frac{d}{dt} }\E_u[\U_i(t)]=\sum_{j\in\cN}W_{j\rightarrow i}\E_u[\varphi(\U_j(t))] -\E_u[\U_i(t)\varphi(\U_i(t))] \\
\quad {} -\alpha\E_u[\U_i(t)]+
\lambda\E_u[\U_i(t)-\bar{\U}(t)].
\end{align*}
Summing over all $i\in\cN$ and then using that $\varphi$ is a non-negative $c$-Lipschitz function such that $\varphi(0)=0$, it follows that
\begin{eqnarray*}
{\frac{d}{dt} }\E_u[\bar{\U}(t)]\leq (c\alpha^*-\alpha)\E_u[\varphi(\bar{\U}(t))].
\end{eqnarray*}
Therefore applying the Grownwall's lemma to the inequality above we conclude the proof.
\end{proof}

Even assuming that the firing rate $\varphi$ satisfies only the Assumption \ref{a1}, we claim that for any fixed numbers of neurons, the presence of leak currents is a necessary and sufficient condition for the extinction of the process.
In fact, we shall prove a stronger result.  It states that, if there is leakage, there will be only a finite numbers of spikes eventually almost surely.
On the other hand, it is shown that, excluding the trivial initial configuration, the system is Harris-ergodic when there is no leakage.
In particular, this results generalize Theorem~\ref{EvaStopSpik} above.

In order to state our main result, we need to introduce some extra notation.
For each neuron $i\in\cN$, let $\T^i_1=\inf\{s>0 : \U_i(s)=0\}$ be the first spiking time of neuron $i$ and for each $k\geq 2$, let $\T^i_k=\inf\{s>T^i_{k-1} : \U_i(s)=0\}$ be the $k$-th spiking time of neuron $i$. Then, the first and the $k$-th spiking time of the system are defined respectively by
\begin{equation}
\label{defdisp}
T_1=\inf\limits_{i\in\cN}T^i_1\quad \mbox{and} \quad T_k=\inf\limits_{i\in\cN,m\geq1}\{T^i_m>T_{k-1}\}, \ \ k\geq 2.
\end{equation}
Our main theorem is given below.
\begin{thm}
\label{UstopSpiking}
Let $(\U(t))_{t\geq 0}$ be the Markov process starting whose the infinitesimal generator is given by (\ref{geradorU}) and $(\T_k)_{k\geq 1}$ be as defined in (\ref{defdisp}). Assume that the spiking rate $\varphi$ is differentiable at $0$ and satisfies 
Assumption \ref{a1}.  Then for any $\alpha>0$, $\lambda\geq 0$ and $u\in \R_+^N$,
$$
\P_u\Bigl(	\sum_{k\geq 1}\one{\{T_k<\infty\}}<\infty\Bigr)=1.
$$
\end{thm}

\begin{cor}
Under the same hypothesis of Theorem \ref{UstopSpiking}, for all $i\in\cN$ and $u\in \R_+^N$, it holds
\begin{equation*}
\lim_{t\rightarrow +\infty}\U_i(t)=0 \quad \P_u\mbox{-a.s.}
\end{equation*}
In particular, the delta of Dirac  at point $0^N$, $\delta_{0^N}$, is the unique invariant measure for the process in the presence of currents.
\end{cor}

It remains to analyse what happens in the long-run behavior of the system in the absence of the leakage. For that sake, we shall assume also that each neuron influences and it is influenced by at least one other neuron. Formally, 
\begin{ass}
\label{ass:grafoInter}
For each neuron $i\in\cN$ there exist at least two neurons $j,j\p\in\cN\setminus\{i\}$ such that $W_{i\to j}>0$ and $W_{j\p \to i}>0$.
\end{ass}
In absence of leakage and under Assumption \ref{ass:grafoInter} we have the following result.
\begin{thm}
\label{med_inv_semleaking}
Let $(\U(t))_{t\geq 0}$ be the Markov process  whose the infinitesimal generator is given by (\ref{geradorU}) with $\alpha=0.$ 
Under Assumptions \ref{a1} and \ref{ass:grafoInter}, for all $u\in \R_+^N\setminus\{0^N\}$ the process $(\U(t))_{t\geq 0}$ with $\U(0)=u$ is Harris-ergodic. In particular, the process $(\U(t))_{t\geq 0}$ does not go extinct.
\end{thm}

\section{Proof of Theorem \ref{UstopSpiking}}
\label{ProvaUstop}
First of all, observe that, from the equation (\ref{determotion}), for any time $t \in [T_n,T_{n+1})$,
\begin{equation}
\label{equacaoU_i(t)}
\U_i(t) =\U_i(\T_n)e^{-(\alpha +\lambda)(t-\T_n)}+ \bU(\T_n)e^{-\alpha (t-\T_n)}\big(1-e^{-\lambda(t-T_n)}\big).
\end{equation}
We shall explore this equation many times. We start giving a lower bound to the probability that no spikes occur when the system starts from initial configurations bounded.
\begin{prop}
\label{tau1infin<r}
Let $(\U(t))_{t\in \R}$ be the Markov process whose the generator is given by (\ref{geradorU}) and $\T_1$ as defined in (\ref{defdisp}). Suppose that the spiking rate $\varphi$ is differentiable at $0$ and satisfies 
Assumption \ref{a1}. For any $r>0$, if $\alpha>0$, then 
$$
\inf_{u\in \bar{B}(0,r)}\P_u(\T_1=\infty)\geq e^{-\frac{N}{\alpha}\int_{0}^{r}{\frac{\varphi(v)}{ v}}dv} >0,
$$
where $\bar{B}(0,r)=\{u\in \R_+^N: ||u||\leq r\}$ is the ball of radius $r$ and center $0^N.$ 
\end{prop}
\begin{proof}
By the equation (\ref{equacaoU_i(t)}), we have, for all $0\leq t<\T_1$, the following inequality
\begin{eqnarray}
\label{2r}
\|\U(t)\| 
\le    r[e^{-(\alpha +\lambda)t} + e^{-\alpha t} (1-e^{-\lambda t})]
=  re^{-\alpha t}.
\end{eqnarray}
Using the inequality  above and the non-decreasing assumption on $\varphi$, we have that for any $u\in \bar{B}(0,r)$
\begin{eqnarray}
\P_u(\T_1>t) &\geq &  
\exp\left\lbrace -\int_0^t N \varphi(\|\U(s)\|)ds \right\rbrace \nonumber \label{tau>t_U<r}
		\ge \exp\left\lbrace {-\frac{N}{\alpha}\int_{re^{-\alpha t}}^{r}{\frac{\varphi(v)}{v}}dv} \right\rbrace. \label{tau>t_U<r2}
\end{eqnarray}
Therefore, taking $t$ to infinite using both Assumption \ref{a1} and the differentiability at 0 of the spiking rate $\varphi$  the result follows.
\end{proof}
We now prove that there exists a sufficiently large radius $r$ such that for any $u\notin \bar{B}(0,r)$ the process returns $\P_u$-a.s to the ball $\bar{B}(0,r)$ in a finite time. For any measurable set $A$ of $\R_+^N$, we define $T_A=\inf\{t>0:\U(t)\in A\}$. 
\begin{prop}
\label{T1>delta_U>r}
Let $(\U(t))_{t\in \R}$ be the Markov process whose the generator is given by (\ref{geradorU}) and assume that $\varphi$ satisfies Assumption \ref{a1}. For any $\alpha>0$, there exists a finite constant $r=r(\alpha,N,\varphi,{\bf W})$ such that for all $u\notin T_{\bar{B}(0,r)}$,
$$
\E_u\big(T_{\bar{B}(0,r)}\big)\leq V(u)<\infty,
$$ 
where for any $u\in\R_+^N$, $V(u)=\sum_{i\in \cN}u_i.$
\end{prop}
\begin{proof}
Indeed, the function $V$ is in the domain of the generator $\mathcal{L}$ defined in \eqref{geradorU} and for any $u\in \R_+^N,$
\begin{eqnarray}
\label{fUtil}
\mathcal{L}V(u)=-\alpha V(u)+\sum_{i\in\cN}\varphi(u_i)(m_i-u_i),
\end{eqnarray}
where for each $i\in\cN$, $m_i=\sum_{j\in \cN}W_{i\to j}.$
Since for any $u\in \R_+^N$, 
$$
\sum_{i\in\cN}\varphi(u_i)(m_i-u_i)\leq \sum_{i\in\cN: u_i\leq m_i}\varphi(m_i)(m_i-u_i)\leq (N-1)\varphi(m)m, 
$$
where $m=\max\{m_i:i\in\cN\},$ then taking $r=\alpha^{-1}((N-1)\varphi(m)m+1)$ we have that for any $u\notin\bar{B}(0,r)$, $\mathcal{L}V(u)<-1.$ 

Now, defining $M_t=V(\U(t))-V(\U(0))-\int_{0}^t\mathcal{L}V(\U(s))ds$, we deduce by applying Dynkin's formula that for any $u\in\R_+^N$,
\begin{eqnarray}
\label{martingale}
\E_u\big[V\big(\U({t\wedge T_{\bar{B}(0,r)}}\big)\big]=V(u)+\E_u\Big[\int_{0}^{t\wedge T_{\bar{B}(0,r)}}\mathcal{L}V(\U(s))ds\Big].
\end{eqnarray}
Now notice that for any $u\notin T_{\bar{B}(0,r)}$, we have that for all $t\geq 0$,
$$
\E_u\Big[\int_{0}^{t\wedge T_{\bar{B}(0,r)}}\mathcal{L}V(\U(s))ds\Big]\leq -\E_u[t\wedge T_{\bar{B}(0,r)}],
$$ 
which, together with \eqref{martingale}, implies that $\E_u[t\wedge T_{\bar{B}(0,r)}]\leq V(u)$, for  $u\notin T_{\bar{B}(0,r)}$ and all $t\geq 0.$ Since we have $\E_u[t\wedge T_{\bar{B}(0,r)}]\geq t\P_u(T_{\bar{B}(0,r)}>t),$ it follows, taking $t\to\infty$, that $\P_u(T_{\bar{B}(0,r)}<\infty)=1$ which implies that $\E_u[T_{\bar{B}(0,r)}]\leq V(u)$ for all $u\notin T_{\bar{B}(0,r)}$, concluding the proof.
\end{proof}
The proof of Theorem \ref{UstopSpiking} is now easy.
\begin{proof}[of Theorem \ref{UstopSpiking}] 
Indeed, Proposition \ref{T1>delta_U>r} implies that the returning times to $\bar{B}(0,r)$ is $\P_u$-a.s finite for any $u\notin \bar{B}(0,r)$, while Proposition \ref{tau1infin<r} provides a positive lower bound  for the probability of never spike again no matter the configuration the process reaches in $\bar{B}(0,r)$ at each returning time to $\bar{B}(0,r).$ Thus by a Borel-Cantelli-like argument we conclude the proof. 
\end{proof}

\section{Proof of Theorem \ref{med_inv_semleaking}}
\label{ProvaExisteMedInv}

To simplify the proof of Theorem \ref{ProvaExisteMedInv} we shall split it  into several steps. The main part of the argument is to find a {\it recurrent regeneration set} $B$ in sense that

\begin{enumerate}
\item if we write $T^+_B=\inf\{t>T_B:\U(t)\in B\}$,
then $\sup_{u\in B}\E_u(T^+_B)<\infty,$ 
and
\item there exist $T>0$, $\epsilon>0$ and a probability measure $\nu$ on $\R^N_{+}$ such that
$$P_{T}(u,A):=\P_u(\U(T)\in A)\geq \epsilon \nu(A), \ u\in B, $$
for all measurable set $A\in \mathcal{B}(\R^N_{+}).$
\end{enumerate}
Condition (ii) is often called {\em localized Doeblin condition}.
Markov processes with a recurrent regeneration set are called {\it Harris-ergodic.} For such Markov Processes an invariant probability measure always exits
and the law of the processes converges in total variation to its invariant measure as the time diverges, see, among others, \cite{MeyTwe:93} or \cite{Asmussen:03}.

In what follows, for any positive real number $a>0$ we use the notation $R_a(x)$, meaning that there exists a constant $l>0$  such that $|R_a(x)|\leq la$ for all $x.$ We call a function of order $a$ any function $R$ satisfying such condition. Note that we are not specifying the domain in which the function $R$ is defined on. For each $\epsilon>0$ define the following event
$$
M_{\epsilon}=\{T_1<\epsilon, T_k-T_{k-1}<\epsilon, \, \, k=2,\cdots, N \},
$$
and consider again the event
$S=\{S_k=k, \, k=1,\ldots,N\}.$
The first lemma below says that, conditioning on the event $M_{\epsilon}\cap S$, when $\epsilon$ is sufficiently small the process evolves, modulo an error of small order, as in the case without electrical synapses ($\lambda=0$). Before stating this lemma we need to introduce a finite sequence of potential configurations.
Consider the sequence $(v_k)_{k=0,\ldots, N}$ with $v_k\in \R_{+}^N$ given by
\begin{equation}
\label{defv0}
v_0=\Big(\sum_{j=2}^{N}W_{j\rightarrow 1},\sum_{j=3}^{N}W_{j\rightarrow 2}, \ldots, W_{N \rightarrow N-1},0\Big),
\end{equation}
and for $1\leq k \leq N,$  $(v_k)_k=0$ and for $i\neq k,$
$$(v_k)_i=\Big(\sum\limits_{j=i+1}^k W_{j\rightarrow i}\Big)\one_{\{i<k\}}+ \Big((v_0)_i+ \sum_{j=1}^{k}W_{j\rightarrow i}\Big)\one_{\{i>k\}}.$$

\begin{lema}
\label{UiamenosdeEpsilon}
Fix $\delta>0.$ If $\U(0)=u\in B(v_0,\delta)$, then conditioning on $M_{\epsilon}\cap S$, for each $k=1,\cdots, N,$ the following equalities hold:
\begin{enumerate}
\item $\U_i(T_k)=(v_k)_i+\sum\limits_{r=i+1}^{k}\lambda(T_r-T_{r-1})d_i(r-1)+R_{\delta\epsilon}(T^k_1,u)+R_{\epsilon^2}(T_1^k,u),$ for any  $i<k;$
\item $\U_i(T_k)=(v_k)_i+\sum\limits_{r=1}^{k}\lambda(T_r-T_{r-1})d_i(r-1)+R_{\delta}(u)+R_{\delta\epsilon}(T^k_1,u)+R_{\epsilon^2}(T_1^k,u),$
for any $i>k$;
\item $\bU(T_k)=\bar{v}_k+R_{\delta}(u)+R_{\epsilon}(T_1^k,u),\ \mbox{if} \ k<N \ \mbox{and} \ \bU(T_N)=\bar{v}_0+R_{\epsilon}(T_1^N,u),$
\end{enumerate}
where $d_i(m)=\bar{v}_m-(v_m)_i$, $T_0=0$ and $T_1^k=(T_1,\ldots,T_k).$ Furthermore, all the partial derivatives of the remainder functions $R_{\delta\epsilon}(T_1^k,u)$, $R_{\epsilon^2}(T_1^k,u)$ above are either of order $\delta$ or $\epsilon.$
\end{lema}

\begin{proof}
The proof is given by induction on $k$. On the event $M_{\epsilon}\cap S$, we have that $\U_1(T_1)=0$ and for each $i=2,\ldots, N$ and $\U(0)=u\in B(v_0,\delta),$
\begin{eqnarray*}
\U_i(T_1)&=& W_{1\rightarrow i} + \bU(0) + (1-\lambda T_1 +R_{\epsilon^2}(T_1))\big((v_0)_i+R_{\delta}(u)-\bar{\U}(0)\big)\\
&=& (v_1)_i+\lambda T_1d_i(0)+R_{\delta}(u)+R_{\epsilon\delta}(T_1,u) + R_{\epsilon^2}(T_1,u),
\end{eqnarray*}
where in the first equality we have used the expansion series of the exponential function.

Thus, from the expression of $\U_i(T_1)$ above,  we conclude that
\begin{eqnarray*}
\bU(T_1)&=&\bar{v}_1+ \lambda T_1 \frac{1}{N} \sum_{i=2}^Nd_i(0) +R_{\delta}(u)+ R_{\epsilon\delta}(T_1,u)+R_{\epsilon^2}(T_1,u)\\
&=&\bar{v}_1 +R_{\delta}(u)+R_{\epsilon}(T_1,u).
\end{eqnarray*}
In addition, it is easy to check that all remainders functions above have partial derivatives with respect to $T_1$ and all of them are either $R_{\delta}$ or $R_{\epsilon}$ functions. Therefore (i), (ii) and (iii) it is verified for $k=1.$

Now, suppose that (i), (ii) and (iii) hold for some fixed $1<k<N$. As before, on the event $M_{\epsilon}\cap S $, we have   $\U_{k+1}(T_{k+1})= 0$ and, for $i<k+1$, using the inductive hypothesis we have that $\U_i(T_{k+1})$ is equal to
\begin{align*}
W_{k+1\rightarrow i}+\bU(T_k)+(1-\lambda(T_{k+1}-T_k)+R_{\epsilon^2}(T_{k-1},T_k))(\U_i(T_k)-\bU(T_k))
\end{align*}
which can be rewritten as
\begin{align*}
(v_{k+1})_i +\sum_{r=i+1}^{k+1}\lambda(T_r-T_{r-1})d_i(r-1)+ R_{\delta \epsilon}(T_1^{k+1},u) +R_{\epsilon^2}(T_1^{k+1},u).
\end{align*}
Using the same argument for the case when $i>k+1$, we get \textit{(ii)} for $k+1$. Using  again the inductive hypothesis and looking at the expression written in the first equality of the membrane potential, it is readily seen that the remainder functions possesses partial derivatives with to $T_l$, for $l=1,\ldots, k+1$ and they are either of order $\delta$ or $\epsilon.$

Finally, summing $\U_i(T_{k+1})$ over all neurons $i=1,\ldots, N$
\begin{eqnarray*}
\bU(T_{k+1})&=& \bar{v}_{k+1}+R_{\epsilon}(T_1^{k+1},u)+R_{\delta}(u).
\end{eqnarray*}
\end{proof}
Note that $v(T_N)=v(0)$, thus, from the previous lemma it follows 
\begin{cor}
\label{UT}
Under the same assumptions of Lemma \ref{UiamenosdeEpsilon}, if $T=N\epsilon<T_{N+1}$ then for each $1\leq i\leq N$,  $\U_i(T)$ can be rewritten as
$$
(v_0)_i+\lambda(T-T_N)d_i(0) + \sum\limits_{r=i+1}^{N}\lambda(T_r-T_{r-1}) d_i(r-1)+R_{\delta\epsilon}(T_1^N,u)+R_{\epsilon^2}(T_1^N,u).
$$
\end{cor}

\begin{obs}
\label{representation}
To simplify the notation, we shall denote the map $\gamma^0:M_{\epsilon}\rightarrow \R_+^N$ by $\gamma^0(t_1^N)=(\gamma_1^0(t_1^N),\ldots, \gamma_N^0(t_1^N))$ where, $\gamma_N^0(t_1^N)=(v_0)_N+\lambda(T-t_N)d_N(0) $ and for each $i=1,\ldots, N-1$, 
$$
\gamma_i^0(t_1^N)=(v_0)_i+\lambda(T-t_N)d_i(0) + \sum_{r=i+1}^{N}\lambda(t_r-t_{r-1})d_i(r-1).
$$
By Corollary \ref{UT}, conditioning on the event $M_{\epsilon}\cap S\cap \{T<T_{N+1}\}$, $T=N\epsilon$, we have the following representation for all $\U(0)=u\in B(v_0, \delta),$

$$\U(T)=\gamma^0(T_1^N)+R_{\delta\epsilon}(T_1^N,u)+R_{\epsilon^2}(T_1^N,u),$$ where both
$R_{\delta\epsilon}(T_1^N,u)$ and $R_{\epsilon^2}(T_1^N,u)$ are multivalued functions whose the $L_1$-norms are remainders functions of order $\delta\epsilon$ and $\epsilon^2$ respectively. 

Recall that, for all $k\in\cN$, $m_k=\sum_{j\in\cN}W_{k\to j}$ is the total amount of potential the neuron $k$ insert in the system each time it spikes. Assumption \ref{ass:grafoInter} implies that $m_k>0$ for any $k\in\cN$. Using this and the fact that $\bar{v}_k\geq m_k/N$ we get the following corollary.
\end{obs}

\begin{cor}
\label{jacobian}
For each $u\in B(v_0,\delta)$, the absolute value of the determinant of the Jacobian of map $M_{\epsilon}\ni t_1^N\mapsto \gamma_u(t_1^N)=\gamma^0(t_1^N)+R_{\delta\epsilon}(t_1^N,u)+R_{\epsilon^2}(t_1^N,u)$,  is given by
 $$
 |J\gamma_u(t_1^N)|=\lambda^N\prod_{i=1}^N \bar{v}_i +R_{\epsilon}(t_1^N,u)+R_{\delta}(t_1^N,u),
 $$ which, under the Assumption \ref{ass:grafoInter}, is different from zero for $\delta$ and $\epsilon$ small enough for all $(t_1^N)\in M_{\epsilon}$ and $u\in B(v_0, \delta)$.

\end{cor}

We shall use this representation to show that the set $C=B(v_0,\delta)$ satisfies the localized Doeblin condition (see proposition \ref{minorisation}) with $T=N\epsilon$,  for $\delta$ and $\epsilon$ sufficiently small. Before proving this proposition, we need an extra lemma.

\begin{lema}
\label{gotoB}
Let $f_u(t_1,\ldots,t_N)=f_{u,(T_1,\ldots, T_N) ,( S_1=1,\ldots, S_N=N)}(t_1,\ldots,t_N)$ denote the joint density of $(T_1,\ldots,T_N)$ with  $(S_1,\ldots, S_N)$ restricted to the event $S$, when the starting configuration is $u$.  Under the Assumption \ref{ass:grafoInter}, it holds that for any $0<\delta<(v_0)_1$
there exists a constant $C_1>0$ such that  for all $u\in B(v_0,\delta),$
$$f_u(t_1,\ldots, t_N)\geq C_1, \  \mbox{for} \  (t_1,\ldots, t_N)\in M_{\epsilon}.$$
\end{lema}

\begin{proof}
For convenience of the reader, in this proof we shall use a function to describe the deterministic flow between consecutive spikes. Formally, for any $t\in\R_+$, consider the function $t\mapsto \Phi_{t}=(\Phi^i_t,i\in\cN)\in \R_+^{\cN}$ given by
$$
\Phi^i_t(u)=u_ie^{-\lambda t}+\bar{u}(1-e^{-\lambda t}).
$$

Since $\P_{u}(T_1> t)= \exp\left[{-\int_{0}^t} \sum_{j=1}^N\varphi(\Phi^j_s(u))ds\right]$, we immediately see that the density function of $T_1$ with $S_1=1$ given that $\U(0)=u$ is
\begin{equation}
\label{denT1}
f_{T_1,S_1=1\mid u}(t_1)=\varphi(\Phi^1_{t_1}(u)) \exp\Big[{-\int_{0}^{t_1}} \sum_{j=1}^N\varphi(\Phi^j_s(u))ds\Big],\ \mbox{for} \ t_1\geq 0.
\end{equation}
Since $u\in B(v_0, \delta)$ and $\delta<(v_0)_1$ (we can find such $\delta$ because the Assumption \ref{ass:grafoInter} implies $(v_0)_1>0$), we know that there exist positive constants $c^1_1$ and $c^1_2$ such that $c^1_1<u_1$ and for all $j=1,\ldots, N$, $u_j<c^1_2$, thus $c^1_1/N<\bar{u}<c_2^1$. But then for all $u\in B(v_0, \delta)$ and $j=1,\ldots, N$, we have that $\Phi^j_s(u)< c_2^1,$ for all $0\leq s<T_1$. 
Thus, using that $\varphi$ is non-decreasing, from the previous inequality and the identity (\ref{denT1}) it follows that
$f_{T_1, S_1=1\mid u}(t_1)\geq \varphi(c_1^1)e^{-t_1N\varphi(c_2^1)}.$ 

Now, from the definition of the process one easily sees that, given $T_1=t_1, S_1=1$ and $\U(0)=u$, the density function of the increment $T_2$ with $S_2=2$, is given, for any $t_2>t_1$, by
\begin{equation}
\label{def-densidadeT2S2}
 f_{T_2,S_2=2\mid T_1=t_1, S_1=1,u}(t_2)=\varphi\big(\Phi^ 2_{t_2}(\Delta_1(\Phi_{t_1}(u)))\big) e^{{-\int_{t_1}^{t_2}} \sum_{j=1}^N\varphi\big(\Phi^j_s(\Delta_1(\Phi_{t_1}(u))\big)ds}.
\end{equation}
Note that if we denote $K=\max_{i,j\in\cN}W_{i\to j}$ then, for all $j \in\cN$, 
$$
\Delta_1\big(\Phi_{t_1}(u)\big)_j=W_{1\rightarrow j}+\U^u_j(t^-_1)<K+c^1_2:=c^2_2.
$$

Moreover, we have that $\bU(T_1)>m_1/N$ which is greater than zero by Assumption \ref{ass:grafoInter}. Therefore, for all $t_1<t_2<T_2$, 
$$
\Phi^2_{t_2}(u)>(m_1/N)(1-e^{-\lambda(t_2-t_1)}):=c^2_1,
$$
From these two inequalities, using again the monotonicity of $\varphi$ and that the average potential is constant between successive jumps it follows that there exist positive constants $c_1^2$ and $c_2^2$ such that, for $t_2>t_1$, we have $$f_{T_2, S_2=2\mid T_1=t_1, S_1=1,u}(t_2)\geq \varphi(c^2_1)e^{-(t_2-t_1)N\varphi(c^2_2)}.$$

In this manner we obtain sequences $(c_1^n)_{n=1,\ldots,N-1}$ and $(c_2^n)_{n=1,\ldots,N-1}$ satisfying for $k=2,\ldots, N$, $$f_{T_k,S_k=k\mid T_{k-1}=t_{k-1},S_{k-1}=k-1,\ldots, T_1=t_1, S_1=1,u}(t_k)\geq \varphi(c_1^{k-1})e^{-(t_{k}-t_{k-1})N\varphi(c^{k-1}_2)},$$
where $t_k>t_{k-1}>\ldots>t_1\geq 0.$ Thus, we have that,  over $M_{\epsilon}$, the product of these conditional densities is strictly positive, finishing the proof.
\end{proof}

Now we are in position to show that the set $C=B(v_0,\delta)$ satisfies the localized Doeblin condition with $T=N\epsilon$,  for $\delta$ and $\epsilon$ small enough.
\begin{prop}[Localized Doeblin condition]
\label{minorisation}
Under the Assumption \ref{ass:grafoInter},
for any $u\in B(v_0,\delta)$, there exists a non negative function $h_u$ such that for all measurable sets $A\in\mathcal{B}(\R_N^+\setminus\{0^N\})$, 
$$
P_{N\epsilon}(u,A)\geq \int_{A}h_u(v)dv,
$$
for all $\delta$ and $\epsilon$ sufficiently small.
Moreover, there exist a measurable set $I$, with positive Lebesgue measure, and a constant $C_3>0$ such that  $h_u(v)\geq C_31_I(v)$ for all $u\in B(v_0,\delta).$
\end{prop}
\begin{proof}

For each $u\in B(v_0, \delta)$, as in Corollary \ref{jacobian} let us call $\gamma_u:M_{\epsilon}\rightarrow I_u$ , the map $$\gamma_u(t_1^N)=\gamma^0(t_1^N)+R_{\delta\epsilon}(t_1^N)+R_{\epsilon^2}(t_1^N),$$
where $I_u=\gamma_u(M_{\epsilon}).$
From the corollaries \ref{UT}, \ref{jacobian}  and the remark \ref{representation}, it follows that for each $u\in B(v_0,\delta)$, conditioning to $M_{\epsilon}\cap S$, the random vector $\U^u(T)$ has a density $h_u$, where
$$h_u(v)=\left\lbrace
		\begin{array}{ll}
			f_u(g_u(v))\mid Jg_u(v)\mid,  &  \ \mbox{if} \  v\in I_u 	\\
			0 & \mbox{otherwise}		\
		\end{array},
	\right.
 $$
with $g_u:I_u\rightarrow M_{\epsilon}$ being the inverse of $\gamma_u$. This concludes the first part of the proposition.

The proof of second part is more delicate and requires some work. From Corollary  \ref{jacobian} and Lemma \ref{gotoB}, it suffices to prove that there exists a set $I$ such that $I \subset\cap_{u\in B(v_0,\delta)}I_u.$

Now, consider the event $B_{\epsilon}$ defined by
$$
B_{\epsilon}= S \cap \Big\{(i-1)\epsilon+\frac{\epsilon}{4}<T_i<i\epsilon-\frac{\epsilon}{4}, \ i=1,\cdots,N\Big\},
$$
define $I=\gamma_0(B_{\epsilon})$ and fix $v\in I.$ We want to show that for all $u\in B(v_0,\delta)$ there is an vector $t_1^N=t_1^N(u)=(t_1(u),\ldots,t_N(u))\in M_{\epsilon}$ such that $\gamma_u(t_1^N)=v.$
To this end, we introduce the function $F(s,t_1^N)=v-\gamma_{0}(t_1^N)-s[R_{\delta\epsilon}(t_1^N,u)+R_{\epsilon^2}(t_1^N,u)],$ for $s\in[0,1]$ and $t_1^N\in M_{\epsilon}.$ Note that we need to show the existence  of vector $t_1^N\in M_{\epsilon}$ such that $F(1,t_1^N)=0.$ This means that we need to study the equation $F(s,t_1^N)=0$ with $t_1^N$ as function of $s$. To ease the notation, from now on we will write $t$ instead of $t_1^N.$

Note that by the definition of $I$, there exists  $t_0\in B_{\epsilon}$ such that $F(0,t_0)=0.$ Besides, $t=t(s)$ is solution of the equation $F(s,t)=0$ if and only if it satisfies
\begin{eqnarray*}
\lefteqn{0=-D\gamma _0(t(s))\cdot     {\frac{dt(s)}{ds}}-[R_{\delta\epsilon}(t(s),u)+R_{\epsilon^2}(t(s),u)]}\\
&& \hspace{2cm}-s[DR_{\delta\epsilon}(t(s),u)+DR_{\epsilon^2}(t(s),u)]\cdot{\frac{dt(s)}{ds}}
\end{eqnarray*}
or equivalently,
\begin{eqnarray}
\label{eqexistence}
\lefteqn{\Big[-D\gamma_0(t(s))-s\big(DR_{\delta\epsilon}(t(s),u)+DR_{\epsilon^2}(t(s),u)\big)\Big]\cdot{\frac{dt(s)}{ds}}}\nonumber
\\ && \hspace{4cm}=R_{\delta\epsilon}(t(s),u)+R_{\epsilon^2}(t(s),u),
\end{eqnarray}
where $Df(\cdot)$ stands for the differential operator of $f$.
By corollary \ref{jacobian}, the linear operator inside the brackets is invertible and the function on the right hand side is of order $\delta\epsilon +\epsilon^2$ whose the derivative is of order $\delta+\epsilon.$ Therefore,  it follows that $t=t(s)$ is a solution of (\ref{eqexistence}) if and only if  it is the solution of the ODE of the form
\begin{equation}
{\frac{dt(s)}{ds}}=H(s,t(s)), \ t(0)=t_0,
\end{equation}
where the derivative of $H$ with respect to $t$ exists and it is of order $\delta+\epsilon$. In particular, it is limited and therefore the ODE has unique solution $t=t(s)$ for all $s$ such that $t(s)\in M_{\epsilon}.$ Since $t(s)$ moves as $H$ which is of order $\delta\epsilon+\epsilon^2$ and the initial condition $t_0$ is at a distance of order $\epsilon$ of $M_{\epsilon}$, we have that $t(1)\in M_{\epsilon}$,  decreasing both $\delta$ and $\epsilon$ if necessary.  Hence, we have that $I\subset \cap_{u\in B}I_u.$
\end{proof}
It is not difficult to see that we may replace $B(v_0,\delta)$ in the Proposition \ref{minorisation} by any compact subset $K$ of $\{u\in \R_+^N: \sum_{i\in\cN}u_i>m_{*} \}$ where $m_*:=\min_{i\in\cN}\{m_i\}$. This is the content of the next result.
\begin{prop}
\label{minorisation1}
Let $K$ be any compact subset of $\{u\in R_+^N: \sum_{i\in\cN}u_i>m_{*} \}$ where $m_*:=\min_{i\in\cN}\{m_i\}$. Under the Assumption \ref{ass:grafoInter}, there exist constants $\epsilon>0$ and $T>0$ as well as a probability measure $\nu$ on $\R_+^N$ such that for all measurable sets $A\in\mathcal{B}(\R_N^+\setminus\{0^N\})$ and for any $u\in K$, 
$$
P_{T}(u,A)\geq \epsilon\nu(A).
$$
\end{prop}
\begin{proof}
Indeed, we may proceed as in Lemma \ref{UiamenosdeEpsilon} to deduce that given $\delta>0$ there exists $\epsilon_1>0$ such that on event $M_{\epsilon_1}\cap S\cap\{N\epsilon_1<T_{N+1}\}$, for any $\U(0)=u\in K$, it holds $\U(N\epsilon_1)\in B(v_0,\delta)$. Now notice that the proof of Lemma \ref{gotoB} works in the same way for $K$ instead of $B(v_0,\delta)$ since $K$ is a compact subset of $\{u\in R_+^N: \sum_{i\in\cN}u_i>m_{*} \}$ and $m_{*}>0,$ so that the event $M_{\epsilon_1}\cap S\cap\{N\epsilon_1<T_{N+1}\}$ has positive probability uniformly on the initial configuration $u\in K.$ The result follows then by Proposition \ref{minorisation}.
\end{proof} 

It remains to find a suitable compact set $K$ satisfying condition (i) of a recurrent regeneration set. This is the content of next result.
\begin{prop}
\label{regen}
There exists a positive finite constant $r$ such that 
$$
\sup_{u\in \bar{B(0,r)}}\E_u(T^+_{\bar{B(0,r)}})<\infty.
$$
\end{prop}
\begin{proof}
By equation \eqref{fUtil} (recall now that $\alpha=0$), it is not difficult to see that $\mathcal{L}V(u)<-1$ for all $u\notin \bar{B}(0,r)$ with $r>Nm^*+a$ for some $a>0$. Thus, proceeding as in the proof of Proposition \ref{T1>delta_U>r}, we deduce that $\E_u(T_{\bar{B}(0,r)})\leq V(u)$ for all $u\notin \bar{B}(0,r)$. 

Finally, given $u\in \bar{B}(0,r)$, we have $||\U(T_{\bar{B}^c(0,r)})||\leq r+\max_{i,j}{W_{i\to j}}:=D$, so that by the strong Markov property for any $u\in \bar{B}(0,r)$,
$$
\E_u\big(T^+_{\bar{B}(0,r)}\big)\leq \sup_{\substack{ v\notin \bar{B}(0,r): ||v||\leq D}}\E_v(T_{\bar{B}(0,r)})\leq ND.$$
\end{proof}
The proof of Theorem \ref{med_inv_semleaking} is now easy.
\begin{proof}[of Theorem \ref{med_inv_semleaking}]
Indeed, after the first spike the processes $(\U(t))_{t\geq 0}$  never exits from the set $\{u\in \R_+^N: \sum_{i\in\cN}u_i>m_{*} \}$. Thus, taking $K=\{u\in \R_+^N: \sum_{i\in\cN}u_i>m_{*} \}\cap \bar{B}(r,0)$ with $r$ as in Proposition \ref{regen}, it follows from Propositions \ref{minorisation1} and \ref{regen} that $K$ is a recurrent regeneration set.   
\end{proof}

\section{Final discussion}
\label{FinalDisc}

The present paper presented a new class of stochastic processes  and studied its asymptotic distribution in absence and presence of leak currents, without external stimuli in either cases.
Similar results have been recently obtained by Robert and Touboul in \cite{Robert:14}. Their model considers a constant leakage rate without the gap junction term ($\alpha =1$  \mbox{and}  $\lambda=0$ in our context) where the synaptic weights are positive i.i.d random variables with finite mean. They showed that if $\int_{0}^1 \varphi(x) / x \, dx<+\infty$, then no spike occurs after some finite time. This condition imposed on the spiking rate $\varphi$ is the same type as ours. They also proved that when $\varphi(0)>0$ and the synaptic weights are bounded almost surely, the process admits a unique non trivial invariant measure. In other words, when external stimuli are considered the system remains active almost surely, contrasting with non external stimuli case. We did not study this case, nonetheless it could be treated similarly as we did in Theorem \ref{med_inv_semleaking}

The non-decreasing assumption on the spiking rate $\varphi$ can be weakened. For instance, we could replace this assumption requiring that $\varphi$ is increasing only on the interval $[0,r]$, $\varphi(u)\geq \varphi(r)>0$ for $u\geq r$ and bounded in any compact interval. Under this assumption the proof of Theorem \ref{UstopSpiking} would be the same. In order to prove Theorem \ref{med_inv_semleaking}, we need to check if Lemma \ref{gotoB} works under this weaker condition. 
One can check that under this assumption 
$f_{u,T_1}(t_1)\geq \varphi(c_1)\wedge\varphi(r)\exp\{-t\varphi(c_2)\}>0,$ for all 
$u\in B(v_0,\delta)$, where $\delta$ and $c_1$ as in the proof of Lemma \ref{gotoB} and $c_2$ is the maximum of $\varphi$ over $B(v_0,\delta)$. Proceeding in this way, it is straightforward to see that the proof of Lemma \ref{gotoB} also works.

\section{Acknowledgments}

We are indebted to R.~Cofr\'e for the discussions which originate this work. We thank E.~Presutti for   
helpful suggestions and stimulating discussions; A.~Galves and E.~L\"ocherbach for the discussions which led us to  the proof of Theorem \ref{EvaStopSpik}. 

This article  was produced as part of the activities of FAPESP  Research, Innovation and Dissemination Center for Neuromathematics (grant 2013/07699-0), S.Paulo Research Foundation). A. Duarte was fully supported by a CNPq fellowship (grant 141270/2013-6) and G. Ost was fully supported by a CNPq fellowship (grant 141482/2013-3).

\bibliography{Bibli}{}
\bibliographystyle{plain}
\nocite{GalEva:13}
\nocite{Errico:14}
\nocite{Davis:84}
\nocite{Thieullen:12}
\nocite{Gerstner:2002:SNM:583784}
\nocite{Robert:14}
\nocite{Evafou:14}

\end{document}